\documentclass[12pt]{amsart}
\usepackage{amsmath}

\newcommand{\qf}[1]{\mbox{$\langle #1\rangle $}}
\newcommand{\pff}[1]{\mbox{$\langle\!\langle #1\rangle\!\rangle $}}

\DeclareMathOperator{\ani}{an}
\DeclareMathOperator{\sgn}{sgn}
\DeclareMathOperator{\Br}{Br}

\newcommand{\hn}{\widetilde{u}}
\newcommand{\CC}{{\mathbb C}}

\newcommand{\NN}{{\mathbb N}}
\newcommand{\QQ}{{\mathbb Q}}
\newcommand{\ZZ}{{\mathbb Z}}

\renewcommand{\phi}{\varphi}
\newcommand{\ignore}[1]{}
\newcommand{\sos}[1]{D_{#1}(\infty)}

\theoremstyle{plain}
\newtheorem{thm}{Theorem}[section]
\newtheorem{lem}[thm]{Lemma}
\newtheorem{prop}[thm]{Proposition}
\newtheorem{cor}[thm]{Corollary}

\theoremstyle{definition}

\theoremstyle{remark}
\newtheorem{rem}[thm]{Remark}

\begin{document}
\title[Anisotropic indefinite quadratic forms]{Dimensions of 
anisotropic indefinite quadratic forms II --- The lost proofs}
\author[D.W.~Hoffmann]{Detlev W.~Hoffmann}
\address{School of Mathematical Sciences, University of Nottingham, University Park,
Nottingham NG7 2RD, UK} 
\email{detlev.hoffmann@nottingham.ac.uk}
\date{}
\keywords{quadratic form, Pfister form, Pfister neighbor, real field, 
ordering, strong approximation property, 
effective diagonalization, $u$-invariant, Hasse number, Pythagoras number,
function field of a quadratic form}
\subjclass[2000]{primary: 11E04; secondary: 11E10, 11E81}
\begin{abstract}
Let $F$ be a field of characteristic different from $2$.  The $u$-invariant and the
Hasse number $\hn$ of a field $F$ are classical and important field invariants 
pertaining
to quadratic forms.  These invariants measure the suprema of dimensions
of anisotropic forms over $F$ that satisfy certain additional properties.

We construct various examples of fields with infinite Hasse number and
prescribed finite values of $u$ that satisfy additional properties
pertaining to the space of orderings of the field.    We also
construct to each $n\in\NN$ a real field $F$ such that $\hn(F)=2^{n+1}$
and each quadratic form over $F$ of dimension $2^n+1$ is a Pfister neighbor. 
These results were announced
(without proof) in \cite{h-indef2}.
\end{abstract}
\maketitle

\section{Introduction}\label{sec-intro}

Throughout this paper, fields are assumed to be of characteristic different
from $2$ and quadratic forms over a field are always assumed to be 
finite-dimensional and nondegenerate.   In this article, we prove some
of the results announced without proof in \cite{h-indef2}.  We refer
to that article (and to \cite{l05}) for all terminology used in the present
paper.

In the next section, we construct real fields $F$ with Hasse number $\hn (F)=\infty$
for each possible pair of values $(p,u)$ such that $p(F)=p$ and $u(F)=u$
(where $p(F)$ resp.\ $u(F)$ are the Pythagoras number resp.\ $u$-invariant
of $F$) and such that in addition $F$ satisfies SAP but not the property $S_1$ or
vice versa, $F$ satisfies $S_1$ but not SAP.

Recall from \cite{h-indef2} that a field $F$ is said
to have property $PN(n)$ if each form of dimension $2^n+1$ over $F$ 
is a Pfister neighbor.  It was shown there that all fields $F$ with 
$\hn(F)\leq 2^n$ have property $PN(n)$, and if $F$ is a field with
$PN(n)$, $n\geq 2$, then $u(F)\leq \hn(F)\leq 2^n$ or 
$2^{n+1}\leq u(F)\leq\hn(F)\leq 2^{n+1}+2^n-2$.  We conjecture that each field
with $PN(n)$, $n\geq 2$, satisfies $\hn(F)\leq 2^n$ or $\hn(F)=2^{n+1}$.  
In the third section, we construct to any $n\geq 2$ a real field that 
satisfies $PN(n)$
and $\hn(F)=2^{n+1}$, showing that the conjectured upper bound 
can in fact be realized for real fields.

All our constructions use variations of Merkurjev's method of iterated function
fields.

\section{Fields with finite $u$-invariant and infinite Hasse number}
In \cite[\S 5]{ep}, one finds examples of non-SAP fields $F$
with prescribed $u$-invariant $2^n$, $n\geq 1$.  These examples were
obtained using the method of intersection of henselian fields
(cf. \cite{pr}).  In this section, we will apply Merkurjev's method
of constructing fields with even $u$-invariant and modify it in a way
such that these fields will be real and such that either they
will be non-SAP or they will not have the property $S_1$.   
Since fields with finite hasse number are always SAP and $S_1$, the
fields we contruct will have inifinte Hasse number.
It furthermore illustrates the independence
of the properties SAP and $S_1$.

Let us first recall some well known results and some special cases of 
Merkurjev's index reduction theorem which we will use in the sequel.
We refer to \cite{m91}, \cite{t90} for details.  See also 
\cite[Ch.\,V.3]{l05} for basic results on Clifford invariants 
$c(q)\in _2\!\Br(F)$ for quadratic forms $q$ over $F$
and how to compute them, and \cite[Ch.\,X]{l05} for basic results on
function fields $F(q)$ of quadratic forms $q$ over $F$.

\begin{lem}\label{albert}
\begin{itemize}\item[(i)]  Let $Q_i=(a_i,b_i)$, $1\leq i\leq n$,
be quaternion algebras over $F$ with associated norm forms $\pff{a_i,b_i}
\in P_2F$. Let $A=\bigotimes_{i=1}^nQ_i$ (over $F$).
Then there exist $r_i\in F^*$, 
$1\leq i\leq n$, and a form $q\in I^2F$, $\dim q =2n+2$ 
such that $c(q)=[A]\in \Br_2F$ and 
$q = \sum_{i=1}^n x_i\pff{a_i,b_i}$ in $WF$. (We will call such
a form $q$ an {\em Albert form} associated to $A$.) 
Furthermore, if $A$ is not Brauer equivalent to a product of $<n$
quaternion algebras (in particular if $A$ is a division algebra), 
then every Albert form 
associated to $A$ is anisotropic.
\item[(ii)] If $q$ is a form over $F$ with either $\dim q=2n+2$ and 
$q\in I^2F$, or $\dim q=2n+1$, or $\dim q=2n$ and $d_{\pm}q\neq 1$, then
there exist quaternion algebras $Q_i=(a_i,b_i)$, $1\leq i\leq n$,
such that for $A=\bigotimes_{i=1}^nQ_i$ we have $c(q)=[A]$, and
there exists an Albert form $\phi$ associated to $A$ 
such that $q\subset\phi$.
\item[(iii)] If $A$ is a division algebra and if $\psi$ is a form
over $F$ of one of the following types:
\begin{itemize}
\item[(a)] $\dim\psi\geq 2n+3$,
\item[(b)] $\dim\psi =2n+2$ and $d_{\pm}\psi\neq 1$,
\item[(c)] $\dim\psi =2n+2$, $d_{\pm}\psi =1$ and $c(\psi )\neq [A]\in \Br_2F$,
\item[(d)] $\psi\in I^3F$,
\end{itemize}
then $A$ stays a division algebra over $F(\psi )$.
\end{itemize}
\end{lem}

Let us also recall some basic facts on the property SAP and
weakly isotropic forms which we  will use and which are
essentially well known.  Recall that a form $q$ over $F$ is called
weakly isotropic if $n\times q$ is isotropic for some $n\geq 1$
(over nonreal $F$, all forms are clearly weakly isotropic as
$WF=W_tF$),

\begin{lem}\label{SAP-wi} \begin{itemize}\item[(i)]
$F$ is SAP if and only if for every $a,b\in F^*$ the form
$\qf{1,a,b,-ab}$ is weakly isotropic.
\item[(ii)] Suppose that $a,b\in F^*$ are such that
$\qf{1,a,b,-ab}$ is not weakly isotropic.  Let $t\in\sos{F}$.  
Then $\qf{1,a,b,-ab}_{F(\sqrt{t})}$
is not weakly isotropic.
\end{itemize}
\end{lem}
\begin{proof}
(i) See \cite[Satz~3.1]{pr73}, \cite[Th.~C]{elp}.

(ii) Suppose $\qf{1,a,b,-ab}_{F(\sqrt{t})}$ is weakly isotropic.
Then there exists an integer $n\geq 1$ such that 
$n\times\qf{1,a,b,-ab}_{F(\sqrt{t})}$ is isotropic.
The isotropy over
$F(\sqrt{t})$ implies that $n\times\qf{1,a,b,-ab}$ contains a
subform similar to $\qf{1,-t}$ (see, e.g., \cite[Ch.\,VII, Th.\,3.1]{l05}).
Since $t$ is totally positive, it can be written as a sum of, say,
$m$ squares in $F$.  But then $m\times\qf{1,-t}$ is isotropic.
Hence $mn\times\qf{1,a,b,-ab}$ is isotropic and thus $\qf{1,a,b,-ab}$ is
weakly isotropic.
\end{proof}

\begin{cor}\label{SAP-pyth} 
Suppose that $a,b\in F^*$ are such that
$\qf{1,a,b,-ab}$ is not weakly isotropic.
\begin{itemize}\item[(i)]  Let $F_{pyth}$ be the
pythagorean closure of $F$ (inside some algebraic closure of $F$). 
Then $\qf{1,a,b,-ab}_{F_{pyth}}$ is not weakly isotropic.
In particular, if $F$ is not SAP, then $F_{pyth}$ is not SAP.
\item[(ii)] Let $\psi$ be a form over $F$ such that $\psi$ is 
isotropic over $F_{pyth}$.  Then $\qf{1,a,b,-ab}_{F(\psi)}$ is not
weakly isotropic.  In particular, if $F$ is not SAP, then $F(\psi )$ 
is not SAP. This is always the case if 
$\psi$ contains a subform $\tau$, $\dim\tau\geq 2$, such that
$|\sgn_P(\tau )|\leq 1$ for all orderings $P$ of $F$.
\end{itemize}
\end{cor}
\begin{proof}
(i) follows immediately from the previous lemma and the fact
that $F_{pyth}$ can be obtained as the compositum of all extensions
$K/F$  (inside an algebraic closure of $F$) which are of the form
$F=F_0\subset F_1\subset F_2\subset\ldots\subset F_n=K$ for some
$n$, where $F_{i+1}=F_i(\sqrt{1+a_i^2})$ for some $a_i\in F_i$.

(ii) Since $\psi$ is isotropic over $F_{pyth}$, the extension
$F_{pyth}(\psi )/F_{pyth}$ is purely transcendental.  Then
$\qf{1,a,b,-ab}_{F_{pyth}(\psi )}$ is not weakly isotropic
because $\qf{1,a,b,-ab}_{F_{pyth}}$ is not weakly isotropic
and because anisotropic forms (here, $n\times\qf{1,a,b,-ab}_{F_{pyth}}$)
stay anisotropic over purely transcendental extensions.

Now suppose $\psi$ has a subform $\tau$ with $\dim\tau\geq 2$ and
$|\sgn_P(\tau )|\leq 1$ for all orderings $P$ of $F$.
Since $\dim\tau\equiv\sgn_P(\tau )\pmod 2$, we have two cases.
If $\sgn_P(\tau )=0$ for all $P$, then $\tau\in W_tF$.  Hence
$\tau_{F_{pyth}}$ is hyperbolic and $\psi_{F_{pyth}}$ is isotropic.

If $|\sgn_P(\tau )|=1$ for all $P$ (which implies that $\dim\tau$ is odd 
and $\geq 3$), then let $d\in F^*$ such that
$q=\tau\perp\qf{d}\in I^2F$.  It follows readily that in fact
$q=\tau\perp\qf{d}\in I^2_tF$.  Thus, $q_{F_{pyth}}$ is hyperbolic 
and the codimension $1$ subform $\tau_{F_{pyth}}$ is isotropic. 
Again, $\psi_{F_{pyth}}$ is isotropic.
\end{proof}

\begin{thm}\label{u-nonSAP}  Let ${\mathcal N}'$ be the set of
pairs of integers $(p,u)$ such that either $p=1$ and $u=0$
or $u=2n\geq 2^m\geq p\geq 2$ for some integers $m$ and $n$.
Let ${\mathcal N}={\mathcal N}'\cup\{ (p,\infty );\ p\geq 2\ \mbox{or}\ p=\infty\}$.
\begin{itemize}
\item[(i)] If $F$ is a real field, then $(p(F),u(F))\in {\mathcal N}$.
If in addition $I^k_tF=0$ then $p(F)\leq 2^{k-1}$.
\item[(ii)] Let $E$ be a real field and let $(p,u)\in {\mathcal N}$.  
Then there exists a real field extension $F/E$ such that
$F$ is non-SAP, $F$ has property $S_1$ and $(p(F),u(F))=(p,u)$. 
In particular, $\hn (F)=\infty$.
\item[(iii)] Let $E$ be a real field and let $(p,u)\in {\mathcal N}$
such that $p\leq 2^{k-1}$, $k\geq 1$.  Then there exists a  real
field extension $F/E$ such that $F$ is non-SAP, $F$ has property $S_1$,
$I^k_tF=0$ and $(p(F),u(F))=(p,u)$.  In particular, $\hn (F)=\infty$. 
\end{itemize}
\end{thm}
\begin{proof}
(i) Clearly, $u(F)$ is either even or infinite.  It is also
obvious that $p(F)=1$ implies $u(F)=0$.    If $p(F)>2^{\ell-1}$ ($\ell\geq 1$) 
then there exists
an $x\in \sos{F}$ such that $2^{\ell-1}\times\qf{1}\perp\qf{-x}$ is
anisotropic.  This form is t.i.\ and a Pfister
neighbor of $\pff{-1,\ldots ,-1,x}\in P_\ell F$ which is therefore
torsion and anisotropic.  Hence $I^\ell_tF\neq 0$ and $u(F)\geq 2^\ell$.
This yields the claim.

\medskip

(ii) First, let us remark that if $u(F)\leq 2$, then $F$
automatically has property $S_1$. In fact, $S_1$ means that
to each torsion binary form $\beta$ over $F$ there exists an integer
$n\geq 1$ such that $(n\times\qf{1})\perp\beta$ is isotropic.
But if $u(F)\leq 2$, then  $\qf{1}\perp\beta$ is 
isotropic as it is a Pfister neighbor of some torsion $2$-fold
Pfister form which itself is hyperbolic as $I^2_tF=0$.

\smallskip

To realize the value $(p,u)=(1,0)$, 
let $F_0$ be the pythagorean closure of $E$.  Consider
the iterated power series field
$F=F_0(\!(x)\!)(\!(y)\!)$.  By Springer's theorem (cf.
\cite[Ch.~VI, \S 1]{l05}), $u(F)=2^2u(F_0)=0$
and $p(F)=p(F_0)=1$.
Note that we have $W_tF=I_tF=0$.
Furthermore, $F$ is not SAP as $\qf{1,x,y,-xy}$ is not
weakly isotropic.

\smallskip

To get the non-SAP field $F$ with $p(F)=u(F)=2$, let
$F_1=F_0(x,y)$ be the rational function field in two variables.
Note that again $F_1$ is not SAP as $\qf{1,x,y,-xy}$ is not
weakly isotropic.  Let $\phi =\qf{1,-(1+x^2)}$, which is anisotropic 
and torsion as $1+x^2\in \sos{F_1}\setminus F_1^2$.  
We now construct an infinite tower
$F_1\subset F_2\subset\ldots $ such that over each $F_i$, $\phi$ stays
anisotropic and $\qf{1,x,y,-xy}$ will not be weakly isotropic.  

The construction is as follows.
Having constructed $F_i$ with the desired properties,
$i\geq 1$, let $F_{i+1}$ be the compositum of all function fields
of $3$-dimensional t.i.\ forms over $F_i$. Since anisotropic $2$-dimensional
forms stay anisotropic over the function fields of forms of dimension
$\geq 3$ (see, e.g. \cite[Th.\,1]{h-isotr}), $\phi$ will stay
anisotropic over $F_{i+1}$.  By Cor.~\ref{SAP-pyth}, 
$\qf{1,x,y,-xy}$ will not be weakly isotropic over $F_{i+1}$.
Now let $F=\bigcup_{i=1}^{\infty}F_i$.  The above shows that
$\phi_F$ is anisotropic so that in particular $u(F)\geq 2$,
and $\qf{1,x,y,-xy}_F$ is not weakly isotropic so that $F$ is not SAP.  
Let $q\in P_2F\cap W_tF$. Any 
$3$-dimensional subform of $q$ is t.i.\ and thus isotropic by
construction of $F$.  Thus, $q$ is hyperbolic.  In particular,
$I^2_tF=0$ as $I^2_tF$ is generated as an ideal by torsion $2$-fold
Pfister forms (cf. \cite[Th.\,2.8]{el-real}).  By \cite[Prop.\,1.8]{el-u1},
this implies $u(F)\leq 2$ and thus $u(F)=p(F)=2$. Clearly,
$I^2_tF=0$.

\smallskip

To get those values $(p,u)$ of ${\mathcal N}$ with $u\geq 4$, we use a 
construction quite similar to that in the proofs of 
\cite[Th.\,2, Th.\,3]{h-pyth}.

So let $p\geq 2$,
$F_1=F_0(x_1,x_2,\ldots ,y_1,y_2,\ldots )$
be the rational function field in an infinite number of variables
$x_i,y_j$ over $F_0$.  Clearly, $F_1$ is not SAP as, for example,
the form $q=\qf{1,x_1,x_2,-x_1x_2}$ is not weakly isotropic.
Let $a=1+x_1^2+\ldots +x_{p-1}^2$ and let 
$\phi =\qf{\underbrace{1,\ldots ,1}_{p-1},-a}$ which is anisotropic
by a well known result of Cassels (cf. \cite[Ch.\,IX, Cor.\,2.4]{l05}).  
Let $n\geq 2$ and consider the multiquaternion algebra
$$A_n=(1+x_1^2,y_1)\otimes\ldots\otimes (1+x_{n-1}^2,y_{n-1})$$
over $F_1$. Then $A$ is a division algebra over $F_1$ and it will
stay a division algebra over $F_1(\sqrt{-1})$ (see, e.g.
\cite[Lem.\,2]{h-filtr}).  By Lemma~\ref{albert}, there exists
a $2n$-dimensional form $\psi_n$ such that in
$WF_1$ we have $\psi_n = \sum_{i=1}^{n-1}c_i\pff{1+x_{i-1}^2,y_{i-1}}$
for suitable $c_i\in F_1^*$.
Since $1+x_{i-1}^2\in \sos{F_1}$, the forms $\pff{1+x_{i-1}^2,y_{i-1}}$ are
torsion and thus $\psi_n\in I^2_tF_1$.  Furthermore, $\psi_n$ is anisotropic
as $A_n$ is division (this stays true over $F_1(\sqrt{-1})$).  

Let now $n\geq 2$ and $p\geq 2$ be such that $2n\geq 2^m\geq p$
for some integer $m$.
Suppose that $K$ is any real field extension of
$F_1$ such that $q_K$ is not weakly isotropic, $(A_n)_{K(\sqrt{-1})}$
is division and $\phi_K$ is anisotropic.  Consider the following three
types of quadratic forms over $K$:
$${\mathcal C}_1(K)=\{ \qf{\underbrace{1,\ldots ,1}_p,-b}\,|\,b\in \sos{K}\}\ ,$$
$${\mathcal C}_2(K)=\{ \qf{\underbrace{1,\ldots ,1}_{2p}}\perp\beta\,|\,\dim\beta 
=2,\ \beta\in W_tF\}\ ,$$
$${\mathcal C}_3(K)=\{ \alpha\,|\, \alpha\in W_tK,\  \dim\alpha \geq 2n+2\}\ .$$
Let $\rho\in {\mathcal C}_1(K)\cup {\mathcal C}_2(K)\cup {\mathcal C}_3(K)$. 
Then $(A_n)_{K(\rho)(\sqrt{-1})}$ is division so that in particular 
$(\psi_n)_{K(\rho)}$ is anisotropic.  For 
$\rho\in {\mathcal C}_i(K)$, $i=1,2$, this follows as $\rho_{K(\sqrt{-1})}$ is
isotropic (recall that in this case $\qf{1,1}\subset\rho$) and therefore
$K(\rho)(\sqrt{-1})=K(\sqrt{-1})(\rho)$ is purely transcendental over 
$K(\sqrt{-1})$. In the case $\rho\in {\mathcal C}_3(K)$ this is a consequence
of Lemma~\ref{albert}(iii).

Also, $\phi_{K(\rho )}$ is anisotropic.  This follows from \cite[Cor.]{h-pyth}
if $\rho\in {\mathcal C}_1(K)$, and from \cite[Th.\,1]{h-isotr} by comparing
dimensions if $\rho\in {\mathcal C}_i(K)$, $i=2,3$.

$q$ will not be weakly isotropic over $K(\rho )$ by Corollary~\ref{SAP-pyth}.

As before, we now construct a tower of fields $F_1\subset F_2\subset\ldots$
as follows.  Having constructed $F_{i}$, we let $F_{i+1}$ be the compositum
of all function fields of forms in ${\mathcal C}_1(F_i)\cup 
{\mathcal C}_2(F_i)$.  Let $F=\bigcup_{i=1}^{\infty}F_i$.
By the above,  $(\psi_n)_F$ is anisotropic (and torsion), so that
$u(F)\geq 2n$.  On the other hand, torsion forms of dimension $>2n$ will
be isotropic by construction.  Thus $u(F)=2n$.

$\phi_F$ is also anisotropic.  Hence $p(F)\geq p$.  By construction,
all forms in ${\mathcal C}_1(F)$ are isotropic and thus $p(F)=p$.

$q_F$ is not weakly isotropic and therefore $F$ is not SAP.
In particular $\hn (F)=\infty$.

Finally, $F$ has property $S_1$ as all forms in ${\mathcal C}_2(F)$ are 
isotropic by construction.

\smallskip

To obtain the values $(p,\infty )$ with $p\geq 2$, we do the same
construction as before, but this time only with forms in ${\mathcal C}_i(F)$,
$i=1,2$.  This will again yield a non-SAP field $F$ with property
$S_1$ and with $p(F)=p$.  However,
this time we have that $(A_n)_F$ will be a division algebra for each
$n\geq 2$, so that $(\psi_n)_F$ will be an anisotropic torsion  
form of dimension $2n$ for each $n\geq 2$.  In particular, $u(F)=\infty$.

\smallskip

Finally, to obtain $(\infty,\infty)$, construct first a non-SAP field $F^{(1)}$
which is $S_1$ and with $(p(F),u(F))=(2,\infty)$ and anisotropic $2n$-dimensional
torsion forms $\psi_n$, $n\geq 2$ and the t.i. form $q$ that is not weakly isotropic,
as done above.  Then repeat this construction
for $p=4$ starting with $F^{(1)}$ as base field to get a non-SAP field $F^{(2)}$
which is $S_1$ and with $(p(F),u(F))=(4,\infty)$.  Note that in this step, the
forms $\psi_n$ will stay anisotropic over $F^{(2)}$ and $q$ will not become
weakly isotropic.  Thus, we get a tower
$F^{(1)}\subset F^{(2)}\subset F^{(3)}\subset F^{(4)}\subset\ldots$ with
$(p(F^{(i)}),u(F^{(i)}))=(2^i,\infty)$. 
Let $F^{(\infty)}=\bigcup_{i=1}^\infty F^{(i)}$.  

The above shows that
$\psi_n$ will stay anisotropic over $F^{(\infty)}$ for all $n\geq 2$, so 
$u(F^{(\infty)})=\infty$.  Clearly, by construction, $F^{(\infty)}$ will be
$S_1$, and also non-SAP since $q$ will not become weakly isotropic.
Finally, over $F^{(i)}$, since $p(F^{(i)})=2^i$, there exists a sum of 
squares $x_i\in F^{(i)}$ such that
$\mu_i:=(2^i-1)\times\qf{1}\perp\qf{-x_i}$ is anisotropic.  Since
$F^{(i+1)}$ (and subsequently $F^{(m)}$, $m>i$) is obtained by iteratively
taking function fields of forms of dimension $>2^{i+1}$, the anisotropic 
$2^i$-dimensional form $\mu_i$
will stay anisotropic over each $F^{(m)}$, $m>i$ 
(see, e.g., \cite[Th.\,1]{h-isotr}), and thus also over
$F^{(\infty)}$, showing that $p(F^{(\infty)})\geq 2^i$ for each $i$, hence
$p(F^{(\infty)})=\infty$.

\medskip

(iii) If $k\leq 2$ then $I^2_tF=0$ and thus $u(F)\leq 2$.  These cases
have already been dealt with in the proof of (ii).  So suppose that
$k\geq 3$.  We repeat the steps in (ii), but when taking composites
of function fields, we now include also function fields of forms in
$${\mathcal C}_4(K)=\{ \alpha\,|\, \alpha\in I^k_tK,\ \dim\alpha\geq 2^k\}$$
in addition to those in ${\mathcal C}_i$, $1\leq i\leq3$ (resp.
${\mathcal C}_1$, ${\mathcal C}_2$ in the case $u=\infty$).
Since by the Arason-Pfister Hauptsatz, we have that anisotropic forms
in $I^kF$ must be of dimension $\geq 2^k$, we immediately see that
by construction $I^k_tF=0$.

$(A_n)_F$ will still be a division algebra by Lemma~\ref{albert}(iii)
as we only consider in addition function fields of forms in $I^k_t$
with $k\geq 3$.  Thus, $\psi_n$ will be anisotropic as above and
we get again that $u(F)=u$.  Since $\dim\phi =p\leq 2^{k-1}$, it
follows from \cite[Th.\,1]{h-isotr} that $\phi_F$ will still be anisotropic
as we only consider in addition function fields of forms which have
dimension $\geq 2^k$.  We conclude similarly as above that $p(F)=p$.

Using the same reasoning as above, Cororollary~\ref{SAP-pyth} implies that
$q_F$ is not weakly isotropic and therefore $F$ is not SAP, so that
in particular $\hn (F)=\infty$.  Obviously, $F$ will again have
the property $S_1$.
\end{proof}

\begin{rem} In \cite[\S\ 5]{ep}, examples of real fields
$F$ with $u(F)=2^n$ have been constructed for each integer $n\geq 1$
with the property that $u(F(\sqrt{a}))=\infty$ and $p(F(\sqrt{a}))=2$.
$u(F(\sqrt{a}))=\infty$ implies that $F$ is non-SAP by
\cite[Cor.~2.4]{ep}.  It is also indicated how to obtain such a 
field which {\em does not\/} satisfy $S_1$ (resp. certain
properties $S_n$ which generalize $S_1$), see \cite[Rem.~5.3]{ep}.
\end{rem}

We will now construct real SAP fields $F$
such that $\hn (F)=\infty$ and $u(F)=2n$ for a given $n$. 
First, we note that it will be impossible to realize such examples
for all values in ${\mathcal N}$ (cf. Theorem~\ref{u-nonSAP}).

\begin{prop}\label{u=2-SAP} Let $F$ be real and SAP.
If $u(F)\leq 2$ then $u(F)=\hn (F)$.
\end{prop}
\begin{proof}  As remarked in the proof of Theorem~\ref{u-nonSAP},
$u(F)\leq 2$ implies that $F$ has property $S_1$.  Since $F$
is SAP by assumption, we thus have $\hn (F)<\infty$.  
Now $p(F)\leq u(F)\leq 2$, and by \cite[Cor.~3.7, Rem.~3.8]{h-indef2} we
have $u(F)=\hn (F)$.
\end{proof}

\begin{thm}\label{u-nonS1} Let ${\mathcal N}$ be as in
Theorem~\ref{u-nonSAP}.
\begin{itemize}
\item[(i)] If $F$ is a real SAP field with $\hn (F)=\infty$,
then $u(F)\geq 4$ and $(p(F),u(F))\in {\mathcal N}$.  Furthermore,
$I^2_tF\neq 0$.
If in addition $I^k_tF=0$, $k\geq 3$, then $p(F)\leq 2^{k-1}$.
\item[(ii)] Let $E$ be a real field and let $(p,u)\in {\mathcal N}$
with $u\geq 4$.  
Then there exists a real field extension $F/E$ such that
$F$ is SAP, $F$ does not have property $S_1$ and $(p(F),u(F))=(p,u)$.
In particular, $\hn (F)=\infty$.
\item[(iii)] Let $E$ be a real field and let $(p,u)\in {\mathcal N}$
with $u\geq 4$ and such that $p\leq 2^{k-1}$, $k\geq 3$.  
Then there exists a real
field extension $F/E$ such that $F$ is SAP, $F$ does not have property $S_1$,
$I^k_tF=0$ and $(p(F),u(F))=(p,u)$.  In particular, $\hn (F)=\infty$. 
\end{itemize}
\end{thm}
\begin{proof}
(i)  If $I^2_tF =0$, then $u(F)\leq 2$ by \cite[Prop.\,1.8]{el-u1}.
The result now follows from Theorem~\ref{u-nonSAP} and 
Proposition~\ref{u=2-SAP}.

\medskip

(ii) We proceed as in the proof of Theorem~\ref{u-nonSAP}(ii)
for the case $(p,u)\in {\mathcal N}$ and $2n=u\geq 4$, except for the
definition of $F_1$, which now will be the power series field in
one variable $t$ over the field which was denoted by $F_1$ in
the proof of Theorem~\ref{u-nonSAP}(ii): 
$F_1=F_0(x_1,x_2,\ldots ,y_1,y_2,\ldots )(\!(t)\!)$.  We keep the notations
for $A_n$, $\psi_n$, ${\mathcal C}_1(K)$, ${\mathcal C}_3(K)$.
We redefine ${\mathcal C}_2(K)$:
$${\mathcal C}_2(K)=\{\qf{1,1}\otimes\qf{1,x,y,-xy}\,|\,x,y\in K^*\}\ .$$
We construct a tower of fields $F_1\subset F_2\subset\ldots$
as follows.  Having constructed $F_{i}$, we let $F_{i+1}$ be the compositum
of all function fields of forms in ${\mathcal C}_1(F_i)\cup 
{\mathcal C}_2(F_i)\cup{\mathcal C}_3(F_i)$.  
Let $F=\bigcup_{i=1}^{\infty}F_i$.

Exactly as in the proof of Theorem~\ref{u-nonSAP}(ii), it follows that
$(u(F),p(F))=(p,u)$.  It remains to show that $F$ is SAP and does
not have property $S_1$.

Now by construction, for all $x,y\in F^*$ we have that
$\qf{1,1}\otimes\qf{1,x,y,-xy}$ is isotropic.  In particular, each form
$\qf{1,x,y,-xy}$ is weakly isotropic, which shows by Lemma~\ref{SAP-wi}
that $F$ is SAP.

Now let $d=1+x_1^2$ and consider the form $\mu_m=m\times\qf{1}\perp t\qf{1,-d}$
which is anisotropic over $F_1$ by Springer's theorem.
Let $L_1=F_1$ and $L_1'=F_1'=F_0(x_1,x_2,\ldots ,y_1,y_2,\ldots )$.
We now construct a tower of fields $L_1\subset L_2\subset\ldots $
such that $L_i$ will be the power series field in the
variable $t$ over some $L_i'$, $L_i=L_i'(\!(t)\!)$, such that
$F_i\subset L_i$, and $(\mu_m)_{L_i}$ anisotropic for all $m\geq 0$
and all $i\geq 1$.  This then shows that $(\mu_m)_{F_i}$ is anisotropic
for all $m\geq 0$, $i\geq 1$, and therefore $(\mu_m)_F$ will be
anisotropic for all $m\geq 0$.  It follows that the torsion form
$(-t\qf{1,-d})_F$ does not represent any element in $\sos{F}$.
Thus, $F$ does not have property $S_1$.

Suppose we have constructed $L_i=L_i'(\!(t)\!)$. Note that
necessarily $L_i$ is real as $(\mu_m)_{L_i}$ is anisotropic
for all $m\geq 0$.  Let $P_i\in X_{L_i'}$ be any ordering and  
$M_i'$ be the compositum over $L_i'$ of the function fields of all forms
(defined over $L_i'$) in
$${\mathcal C}'(L_i')=\{\alpha\,|\,\mbox{$\alpha$ indefinite at $P_i$,
$\dim\alpha\geq 3$}\}\ .$$
Let $M_i=M_i'(\!(t)\!)$.

Now let $\rho\in {\mathcal C}_1(F_i)\cup {\mathcal C}_2(F_i)
\cup{\mathcal C}_3(F_i)$ and consider $L_i(\alpha )$.
By Springer's theorem, $\rho_{L_i}\cong\beta\perp t\gamma$ where
$\beta$, $\gamma$ are defined over $L_i'$.  Suppose 
$\rho\in {\mathcal C}_1(F_i)$.  Then $\rho\cong p\times\qf{1}\perp \qf{-b}$
with $b\in \sos{L_i}$.  But then, up to a square, $b\in \sos{L_i'}$ and
thus $\rho_{L_i}\in {\mathcal C}'(L_i')$.  Hence, $\rho_{M_i}$ is isotropic
and therefore $M_i(\rho)/M_i$ is purely transcendental.

Suppose $\rho =\qf{1,1}\otimes\qf{1,x,y,-xy}\in {\mathcal C}_2(F_i)$.  
Then either $\rho_{L_i}$ is already defined over $L_i'$, in which case
it is a t.i.\ form of dimension $8$ and thus in
${\mathcal C}'(L_i')$.  Or there exist $a,b\in L_i^{\prime*}$ such that
$\rho \cong\qf{1,1}\otimes\qf{1,a}\perp bt\qf{1,1}\otimes\qf{1,-a}$.
then either $\qf{1,1}\otimes\qf{1,a}$ is indefinite at $P_i$ and thus
in ${\mathcal C}'(L_i')$, or $\qf{1,1}\otimes\qf{1,-a}$ is indefinite
at $P_i$ and thus in ${\mathcal C}'(L_i')$.  In any case, we see that
$\rho_{M_i}$ is isotropic, and again $M_i(\rho)/M_i$ is purely transcendental.

Finally, suppose that $\rho \in {\mathcal C}_3(F_i)$.  Then $\rho_{L_i}\in
W_tL_i$, and if we write $\rho\cong \beta\perp t\gamma$ with $\beta$ and
$\gamma$ defined over $L_i'$, then $\beta\in W_tL_i'$ and 
$\gamma\in W_tL_i'$.  Now $\dim\rho\geq 6$, and hence $\dim\beta\geq 4$
or $\dim\gamma\geq 4$.  Hence $\beta\in {\mathcal C}'(L_i')$ or
$\gamma\in {\mathcal C}'(L_i')$.  As above, we conclude that 
$\rho_{M_i}$ is isotropic and that $M_i(\rho)/M_i$ is purely transcendental.

Now let $N_i$ be the compositum of the function fields of all
forms $\alpha_{M_i}$ with $\alpha\in 
{\mathcal C}_1(F_i)\cup {\mathcal C}_2(F_i)\cup{\mathcal C}_3(F_i)$.
By the above, $N_i/M_i$ is purely transcendental.  Let $B$ be a
transcendence basis so that $N_i=M_i(B)=M_i'(\!(t)\!)(B)$.
We now put $L_{i+1}'=M_i'(B)$ and $L_{i+1}=L_{i+1}'(\!(t)\!)=
M_i'(B)(\!(t)\!)$.  There are obvious inclusions 
$F_{i+1}\subset N_i=M_i'(\!(t)\!)(B)\subset M_i'(B)(\!(t)\!)=L_{i+1}$.
Since $M_i'$ is obtained from $L_i'$ by taking function fields
of forms indefinite at $P_i$, we see that $P_i$ extends to an ordering
on $M_i'$ and thus clearly also to orderings on $L_{i+1}$.  

It remains to show that $\mu_m$ stays anisotropic over $L_{i+1}$.
Now $m\times\qf{1}$ is clearly anisotropic over the real
field  $L_{i+1}'$.   Also, $\qf{1,-d}$, which is anisotropic
over $L_i'$ by assumption, stays anisotropic over $L_{i+1}'$
as $L_{i+1}'$ is obtained by taking function fields of forms of
of dimension $\geq 3$ over $L_i'$ followed by a purely transcendental
extension.  By Springer's theorem, $(\mu_m)_{L_{i+1}}=
(m\times\qf{1}\perp t\qf{1,-d})_{L_{i+1}}$ is anisotropic.

\smallskip

To get the values of type $(p,\infty )$, $(\infty,\infty)$, 
we adjust the above arguments
as in the proof of Theorem~\ref{u-nonSAP}(ii).

\medskip

(iii)  This follows easily by combining the proof of part (ii) above
with that of Theorem~\ref{u-nonSAP}(iii).
We leave the details to the reader.
\end{proof}

\begin{rem} Let $K$ be any real field over $E$ with
$u(K)=2n$ and such that $K$ is uniquely ordered.
For $n\geq 2$, such fields have been constructed in 
\cite[Th.\,2]{h-pyth}. 
The construction there can also readily
be used to get such a $K$ for $n=1$.

Now consider $F=K(\!(t)\!)$, the power series field in one
variable $t$ over $K$.  By Springer's theorem, $u(F)=4n=2u(K)$.
Since $K$ is uniquely ordered, we have that $F$ is SAP
(cf. \cite[Prop.~1]{elp}).  Since $u(K)>0$, $K$ is not 
pythagorean.  So let $d\in \sos{K}\setminus K^{*2}$.  
Then the form $(m\times\qf{1})\perp t\qf{1,-d}$ is anisotropic
for all $m$ (again by Springer's theorem), and since
$t\qf{1,-d}$ is torsion, we see that $F$ does not have
property $S_1$.  Hence $\hn (F)=\infty$.

This rather simple construction yields SAP fields with
$u(F)=4n$ and $\hn (F)=\infty$ for all $n\geq 1$, but it does not
provide examples where $u(F)=4n+2$, $n\geq 1$.  Furthermore,
one checks easily that it will not yield examples of SAP fields
with $\hn (F)=\infty$,  
$u(F)>4$ and $I^3_tF=0$, which do exist by the above theorem.
\end{rem}

\section{Fields with $PN(n)$ and $\hn(F)=2^{n+1}$}
In \cite{b1}, Becher studies fields $F$ that possess an anisotropic form
$\phi$ such that any other anisotropic form over $F$ is a subform of $\phi$.
It can be shown that such a form $\phi$ is then necessarily an $n$-fold Pfister
form for some $n\in \NN_0$ (called {\em supreme Pfister form}), in which case
$F$ is nonreal and $u(F)=\dim\phi=2^n$.  It is clear that any such field will have
property $PN(n-1)$.  A well known example of such a field is 
the iterated power series field $F=\CC(\!(X_1)\!)(\!(X_2)\!)\ldots (\!(X_n)\!)$,
where the supreme Pfister form is given by $\pff{X_1,\ldots,X_n}$.  

This also shows that for any $n\geq 2$, there exist nonreal fields $F$ with property 
$PN(n)$ and $u(F)=2^{n+1}$.

To get real fields with $PN(n)$ ($n\geq 2$) and $u(F)=\hn(F)=2^{n+1}$, consider
the real field $K=\QQ (X_1,\cdots ,X_n)$.  Let $\pi =\pff{2,X_1,\ldots,X_n}$.  
One readily sees that $\pi$ is anisotropic and torsion (since $\pff{2}\cong\qf{1,-2}$
is torsion).  Fix an ordering $P\in X_K$. 
Now consider
$${\mathcal C}=\{ \mbox{field extensions $L$ of $K$ s.t. $P$ extends to $L$ and $\pi_L$
anisotropic}\}$$
Clearly, $K\in {\mathcal C}$, ${\mathcal C}$ is closed under direct limits, and if
$L\in {\mathcal C}$ and $L'$ is a field with $K\subset L'\subset L$, 
then $L'\in {\mathcal C}$.
Then, by \cite[Theorem 6.1]{b1}, there exists a field $F\in {\mathcal C}$ such that for any
anisotropic form $\phi$ over $F$, $\dim\phi\geq 2$, one has that 
$F(\phi)\notin {\mathcal C}$.  We claim that $F$ has a unique ordering (which extends $P$),
that $F$ has $PN(n)$ and that $u(F)=\hn (F)=2^{n+1}$.

Now by construction, $F$ is real with an ordering $P'$ extending $P$.
Suppose there exists $Q\in X_F$ with $Q\neq P'$.  Let $a\in F$ such that $a>_{P'}0$ and
$a<_Q0$, and consider $q\cong (2^{n+1}\times\qf{1})\perp\qf{-a}$.  Then $q$ is 
anisotropic as it is positive definite at $Q$, and $P'$ (and thus $P$) extends to 
$F(q)$ as $q$ is indefinite at $P'$.  However, since $\dim q=2^{n+1}+1>2^{n+1}=\dim\pi$,
$\pi$ stays anisotropic over $F(q)$.  Hence $F(q)\in {\mathcal C}$, a contradiction.
Thus, $X_F=\{ P'\}$.

In particular, since $\pi_F$ is torsion and anisotropic, we have $u(F)\geq 2^{n+1}$.
Suppose $\hn (F) >2^{n+1}$.  Then there exists an anisotropic t.i. form $\tau$ 
with $\dim\tau > 2^{n+1}$.  A similar reasoning as above shows that 
$F(\tau)\in {\mathcal C}$, again a contradiction.  Hence $\hn(F)\leq 2^{n+1}$
and we have $u(F)=\hn (F)=2^{n+1}$.  

Now let $\psi$ be any form of dimension $2^n+1$ over $F$.  If $\psi$ is isotropic,
it is easily seen to be a Pfister neighbor of the hyperbolic $(n+1)$-fold 
Pfister form.  So assume that $\psi$ is anisotropic.
Suppose first that $\psi$ is t.i. and consider $\rho =(\pi_F\perp-\psi)_{\ani}$.
Then $2^n-1\leq\dim\rho$.  If $\dim\rho>2^n-1$ then $\dim\rho\geq 2^n+1=\dim\psi$
and $|\sgn_{P'}\rho |=|\sgn_{P'}\psi|\leq 2^n-1$, so in particular $\rho$ is t.i. and thus
$P'$ extends to $F(\rho)$.  Since we cannot have $F(\rho)\in {\mathcal C}$, we must 
therefore have that $\pi_{F(\rho)}$ is isotropic and hence hyperbolic, so
$\rho$ is similar to a subform of $\pi_F$.  Thus, there exists $x\in F^*$ and a form
$\gamma$, $\dim\gamma \leq 2^n-1$ with $x\pi_F\cong\rho\perp\gamma$.  Thus, in $WF$, we get
$x\pi_F =\pi_F\perp-\psi\perp\gamma$.  
But $\pi_F\perp -x\pi_F\in P_{n+2}F$ is torsion, therefore
isotropic since $u(F)=2^{n+1}$ and thus hyperbolic (this actually shows
that $x\pi_F\cong\pi_F$ for any $x\in F^*$).  Hence, we have $\psi =\gamma$
in $WF$ with $\psi$ anisotropic and $\dim\psi >\dim\gamma$, a contradiction.
It then follows that $\dim\rho=2^n-1$ and therefore $\pi_F\cong\rho\perp\psi$, showing
that $\psi$ is a Pfister neighbor of $\pi_F$.

Now suppose that $\psi$ is definite at the unique ordering $P'$ of $F$.  
After scaling, we may assume that $\psi$ is positive definite.  Let 
$\sigma = 2^{n+1}\times\qf{1}\in P_{n+1}F$.  If $\psi$ is a subform of $\sigma$
then it is a Pfister neighbor and we are done.  So suppose that $\psi$ is not a 
subform of $\sigma$ and let $\eta \cong (\sigma\perp -\psi)_{\ani}$.  We then
have that $\dim\eta \geq 2^n+1$ whereas $\sgn_{P'}\eta =2^n-1$.  In particular,
$\eta$ is t.i., and $P'$ extends to $F(\eta)$.  But $F(\eta)\notin {\mathcal C}$,
so we must have that $\pi_{F(\eta)}$ is isotropic and hence hyperbolic, and
as above we have that $\pi_F\cong\eta\perp\delta$ for some form $\delta$ with
$\dim\delta \leq 2^n-1$.  In $WF$, we thus get 
$\sigma\perp -\pi_F = \psi\perp -\delta\in I^{n+1}F$.  Now since 
$\dim\psi =2^n+1\geq \dim\delta +2$,
we have that $\psi\perp -\delta$ is of dimension $\leq 2^{n+1}$ but not hyperbolic.
By the Arason-Pfister Hauptsatz, 
we necessarily have that $\dim\delta =2^n-1$ and $\psi\perp -\delta\in GP_{n+1}F$,
so $\psi$ is a Pfister neighbor, showing that $F$ has property $PN(n)$.

Let us finally remark that in this example, the proof shows that 
$\pi_F$ is the unique anisotropic 
torsion $(n+1)$-fold Pfister form over $F$, and that there are two anisotropic
(positive definite) $(n+1)$-fold Pfister forms, namely $\sigma$ and
$(\sigma\perp -\pi)_{\ani}$.  This also implies that 
$I^{n+1}F/I^{n+2}F\cong\ZZ/2\times\ZZ/2$.\qed

\end{document}